\documentclass[12pt,a4paper]{article}

\usepackage[utf8]{inputenc}
\usepackage{amssymb, amsthm, amsmath}
\usepackage[english]{babel}
\usepackage[T1]{fontenc}
\usepackage{graphicx}
\usepackage{mathtools}
\usepackage[font=small,labelfont={normal, bf}]{caption}
\usepackage{subcaption}
\usepackage{xcolor}

\usepackage[cm]{fullpage}

\usepackage{autonum}

\makeatletter
\newcommand{\@giventhatstar}[2]{#1\;\middle|\;#2)}
\newcommand{\@giventhatnostar}[3][]{#1#2\;#1|\;#3#1}
\newcommand{\giventhat}{\@ifstar\@giventhatstar\@giventhatnostar}
\makeatother

\newtheorem{theorem}{Theorem}
\newtheorem{proposition}{Proposition}
\theoremstyle{definition}
\newtheorem{definition}{Definition}
\newtheorem{corollary}{Corollary}

\title{Error of discretization of Caputo fractional derivative in weighted spaces}
\author{
\L{}ukasz P\l{}ociniczak$^*$, \and Hubert Woszczek\thanks{Faculty of Pure and Applied Mathematics, Wroclaw University of Science and Technology, Wyb. Wyspia\'nskiego 27, 50-370 Wroc{\l}aw, Poland}$^{*,}$\thanks{\underline{Corresponding author:} \texttt{hubert.woszczek@pwr.edu.pl}}
}
\date{}

\begin{document}
\maketitle

\begin{abstract}
	We establish uniform error bounds of the L1 discretization of the Caputo fractional derivative of the function from the weighted Sobolev space with weight belonging to the Mucknenhoupt class. We present how our framework works for several examples of weight, which belong to the Muckenhoupt class. As and application, we show the convergence of the L1 scheme for the Fractional ODE. Finally, we verify the theoretical results with numerical illustrations.
\end{abstract}



\section{Introduction}
The Caputo derivative is defined for sufficiently smooth functions by
\begin{equation}\label{eq:1.1}
D^\alpha y(t) = \frac{1}{\Gamma(1-\alpha)} \int_0^t \frac{y'(s)}{(t-s)^\alpha} ds \ , \quad 0 < \alpha < 1.
\end{equation}
Equivalence between formulations can be recovered by integrating by parts. The most prominent methods used to discretize the Caputo derivative are derived or generalized from Convolution Quadrature (utilizing the convolutional form of the Caputo derivative, \cite{Lubich2004}), Gr\"{u}nwald-Letnikov (a generalization of finite differences, \cite{Gorenflo2007}), or the L1 scheme (based on piecewise linear approximation, \cite{Stynes2022}). The latter can be expressed as
\begin{equation} \label{eq:1.2}
    \delta_\tau^\alpha y(t_n) := \frac{\tau^{-\alpha}}{\Gamma(2-\alpha)} \left( y(t_n) - b_{n-1}y(0) - \sum_{i=1}^{n-1} (b_{n-i-1} - b_{n-i})y(t_i) \right),
\end{equation}
where $b_i = (i+1)^{1-\alpha} - i^{1-\alpha}$. Through this scheme, the Caputo derivative is approximated on the discrete grid $t_n = n\tau \leq T$ (with $t_0 = 0$), where $\tau > 0$ is the time step and $T > 0$ is the maximal time. 

The L1 method has many favorable properties that make it versatile in the design of numerical methods for various applications. Due to its convenient properties, the L1 method was analyzed in various functional settings. The historically first proof of the accuracy of the discretization was obtained for functions belonging to space $C^2([0,T])$ resulting in order $2-\alpha$ (see, for example, \cite{LiCai2019}, Theorem 4.1). In this setting, it is also possible to find the explicit asymptotic form of the best error constant $\overset{\sim}{C}_{\alpha,n}\sim-\zeta(\alpha-1)/\Gamma(2-\alpha)$, where $\zeta$ is the Riemann zeta function (see \cite{Plociniczak2023}). Recently, the trunction error for the function belonging to the H\"older space $C^{k, \beta}([0,T])$, with $k=0,1$ and $0\le\beta\le1$ with $k+\beta>\alpha$ was analyzed in \cite{delTesoPłociniczak2025}. The resulting order was equal to $k+\alpha-\beta$. One can also use graded meshes to maintain an optimal convergence rate of $2-\alpha$, simultaneously relaxing the functional setting to space $C^{\alpha}([0,T])$, with $y^{(m)}\sim t^{\alpha-m}$ for $m=0,1,2$ (see \cite{Stynesetal2017, Kopteva2019}). In this paper, we establish uniform bounds on the discretization error of the L1 method for functions belonging to a weighted Sobolev space $W^{2,p}_\omega(0,T)$. To formally define this functional setting, we first introduce the weighted space $L^p$.
\begin{definition}
    Let $\omega: (0,T) \to [0, \infty)$ be a locally integrable weight function and $1 \leq p < \infty$. The weighted space $L^p$, denoted $L^p_\omega(0,T)$, is equipped with the norm
\begin{equation}
    \|y\|_{L^p_\omega} := \left( \int_0^T |y(t)|^p \, \omega(t) dt \right)^{1/p}.
\end{equation}
\end{definition}
Building upon this, we define the corresponding weighted Sobolev space.
\begin{definition}
The weighted Sobolev space $W^{s,p}_\omega(0,T)$ is defined as
\begin{equation}
    W^{s,p}_\omega(0,T) := \left\{ y \in L^p_\omega(0,T) : y', y'', \ldots, y^{(s)} \in L^p_\omega(0,T) \right\},
\end{equation}
endowed with the norm
\begin{equation} \label{eq:W1p_norm}
    \|y\|_{W^{s,p}_\omega} := \left( \int_0^T \left( \sum_{k=0}^s|y^{(k)}(t)|^p \right) \omega(t) dt \right)^{1/p}.
\end{equation}
\end{definition}
Finally, we require the weight function to satisfy a specific integrability condition, encapsulated by the Muckenhoupt class.
\begin{definition}
For $1 < p < \infty$, a weight $\omega$ belongs to the Muckenhoupt class $A_p$ if the following condition holds
\begin{equation} \label{eq:Ap_condition}
    [\omega]_{A_p} := \sup_{J \subset (0,T)} \left( \frac{1}{|J|} \int_J \omega(t) dt \right) \left( \frac{1}{|J|} \int_J \omega(t)^{-\frac{1}{p-1}} dt \right)^{p-1} < \infty,
\end{equation}
where the supremum is taken over all sub-intervals $J$ of $(0,T)$.
\end{definition}
It is a well-known phenomenon in the theory of fractional calculus that solutions to time-fractional differential equations typically exhibit a weak singularity at the initial time $t=0$, even when the source term and the initial data are smooth. For example, the solution often contains terms $y(t) \sim t^\alpha$ near the origin. While the function itself remains continuous, its derivatives become singular. Due to this singular behavior, the second derivative $y''(t)$ frequently fails to be integrable over an open interval, which means that the exact solution $y$ does not belong to the standard Sobolev space $W^{2,p}(0,T)$ or the space of smooth functions $C^2[0,T]$. Consequently, traditional error analyses that rely on the assumption $y \in C^2[0,T]$ or $y \in W^{2,p}(0,T)$ are practically restrictive and often inconsistent with the true regularity of fractional differential equations. To resolve this limitation, it is necessary to introduce a weight function $\omega(t)$ that controls the singularity near the origin. Thanks to that, we can guarantee that $y \in W^{2,p}_\omega(0,T)$. The Muckenhoupt $A_p$ class is specifically chosen here because it provides the precise harmonic analysis framework required to bound the Riemann-Liouville fractional integral operators in weighted spaces. Utilizing $A_p$ weights allows us to perform a robust numerical analysis that is entirely consistent with the physically realistic, non-smooth behavior of fractional-order models.

\section{Main results}
We start with the following proposition, which justifies the well-posedness of the Caputo derivative in our spaces. 
\begin{proposition}[Well-Definedness of the Caputo Derivative]
Let $0 < \alpha < 1$ and $1 < p < \infty$. Assume that the weight $\omega$ belongs to the Muckenhoupt class $A_p$ on $(0,T)$. Then, the Caputo fractional derivative operator $D^\alpha$ is a bounded linear operator from $W^{1,p}_\omega(0,T)$ to $L^p_\omega(0,T)$. That is, there exists a constant $C > 0$ such that for all $y \in W^{1,p}_\omega(0,T)$
\begin{equation}
    \| D^\alpha y \|_{L^p_\omega} \leq C \| y \|_{W^{1,p}_\omega}.
\end{equation}
\end{proposition}
\begin{proof}
It is well-known that the Caputo derivative can be expressed as the Riemann-Liouville fractional integral of order $1-\alpha$ applied to the first derivative $y'$. Let $\beta = 1 - \alpha$. We denote the Riemann-Liouville fractional integral operator by 
\begin{equation}
    I^\beta f(t) := \frac{1}{\Gamma(\beta)} \int_0^t (t-s)^{\beta-1} f(s) ds.
\end{equation}
Thus, we can write $D^\alpha y(t) = I^{1-\alpha}(y')(t)$. To prove the proposition, it suffices to show that the operator $I^{1-\alpha}$ is bounded on the weighted space $L^p_\omega(0,T)$. It is a classical result in harmonic analysis (see, e.g., \cite{MuckenhouptWheeden1974}) that the Riemann-Liouville fractional integral operator $I^\beta$ is bounded on $L^p_\omega(0,T)$ provided that the weight $\omega$ satisfies the Muckenhoupt $A_p$ condition. Therefore, we have
\begin{equation}
    \| I^{1-\alpha} y' \|_{L^p_\omega} \leq C \| y' \|_{L^p_\omega},
\end{equation}
where $C$ is a constant depending on $p$, $\alpha$, and the $A_p$ characteristic of $\omega$. Since $\| y' \|_{L^p_\omega} \leq \| y \|_{W^{1,p}_\omega}$ by the definition of the Sobolev norm, we conclude
\begin{equation}
    \| D^\alpha y \|_{L^p_\omega} = \| I^{1-\alpha} y' \|_{L^p_\omega} \leq C \| y \|_{W^{1,p}_\omega},
\end{equation}
and the proof is complete. 
\end{proof}
Next, we state the main result.

\begin{theorem}\label{th1}
Let $\omega \in A_p$ and $y \in W^{2,p}_\omega(0,T)$. The uniform error of the L1 scheme satisfies
\begin{equation}
    |D^\alpha y(t_n) - \delta_\tau^\alpha y(t_n)| \le C \tau^{2-\alpha-1/p} \Lambda(\omega, \tau) \|y\|_{W^{2,p}_\omega},
\end{equation}
where
\begin{equation}
    \Lambda(\omega, \tau) = \sup_{0 \le j \le N-1} \tau^{-1/q} W_j = \sup_{j} \left( \frac{1}{\tau} \int_{t_j}^{t_{j+1}} \omega(s)^{-\frac{1}{p-1}} ds \right)^{\frac{p-1}{p}}.
\end{equation}
\end{theorem}

\begin{proof}
We consider a uniform grid $t_n = n\tau$ for $n=0, \dots, N$, with step size $\tau = T/N$. The L1 scheme is given by replacing $y'(s)$ with the backward difference quotient $u_j$:
\begin{equation} \label{eq:L1}
    \delta_\tau^\alpha y(t_n) = \frac{1}{\Gamma(1-\alpha)} \sum_{j=0}^{n-1} u_j \int_{t_j}^{t_{j+1}} (t_n - s)^{-\alpha} ds,
\end{equation}
where $u_j = \frac{y(t_{j+1}) - y(t_j)}{\tau}$. Let us define the seminorm
\begin{equation}
    [y]_{W^{2,p}_\omega} := \|y''\|_{L^p_\omega}.
\end{equation}
Since $[y]_{W^{2,p}_\omega} \le \|y\|_{W^{2,p}_\omega}$, we will derive the bounds in terms of this seminorm. Let us denote $R_n(y) := |D^\alpha y(t_n) - \delta_\tau^\alpha y(t_n)|$. Subtracting \eqref{eq:L1} from the exact derivative, we have
\begin{equation} \label{eq:error_rep}
    R_n(y) = \frac{1}{\Gamma(1-\alpha)} \left|\sum_{j=0}^{n-1} \int_{t_j}^{t_{j+1}} (t_n - s)^{-\alpha} (y'(s) - u_j) ds \right|.
\end{equation}
Since $y'$ is an absolutely continuous function, using the Taylor series expansion with the integral form of the remainder, for all $s \in (t_j, t_{j+1})$, we have
\begin{equation}
    |y'(s) - u_j| = \left|y'(s) - \frac{1}{\tau}\int_{t_j}^{t_{j+1}} y'(\xi) d\xi\right| = \left|\tau^{-1} \int_{t_j}^{t_{j+1}} \int_{\xi}^s y''(\eta) d\eta d\xi\right| \le \int_{t_j}^{t_{j+1}} |y''(\eta)| d\eta.
\end{equation}
Let $Y_j := \int_{t_j}^{t_{j+1}} |y''(\eta)| d\eta$. Thus, $\sup_{s \in (t_j, t_{j+1})} |y'(s) - u_j| \le Y_j$. Since the only singularity occurs in the term $j = n - 1$ with $n \ge 1$, this term has to be estimated separately. By direct calculation, we have
\begin{align}
    \left|\int_{t_{n-1}}^{t_{n}} (t_n - s)^{-\alpha} (y'(s) - u_{n-1}) ds\right| &\le \int_{t_{n-1}}^{t_n} (t_n - s)^{-\alpha} |y'(s) - u_{n-1}| ds \nonumber \\
    &\le Y_{n-1} \int_{t_{n-1}}^{t_n} (t_n - s)^{-\alpha} ds = \frac{\tau^{1-\alpha}}{1-\alpha} Y_{n-1}.
\end{align}
We estimate the remaining part as follows. Since $\int_{t_j}^{t_{j+1}} (y'(s) - u_j) ds = 0$, and $(t_n - t_{j+1})^{-\alpha}$ are constant, we have
\begin{align}
    \left|\int_{t_j}^{t_{j+1}} (t_n - s)^{-\alpha} (y'(s) - u_j) ds\right| &= \left|\int_{t_j}^{t_{j+1}} \left[ (t_n - s)^{-\alpha} - (t_n - t_{j+1})^{-\alpha} \right] (y'(s) - u_j) ds \right| \nonumber \\
    &\le \int_{t_j}^{t_{j+1}} |(t_n - s)^{-\alpha} - (t_n - t_{j+1})^{-\alpha}| |y'(s) - u_j| ds.
\end{align}
Let $k(s) = (t_n-s)^{-\alpha}$. From the Mean Value Theorem, we have $|k(s) - k(t_{j+1})| = |k'(\xi)||s-t_{j+1}|$ for some $\xi\in[t_j, t_{j+1}]$. Since $k'(s) = \alpha(t_n-s)^{-\alpha-1}$, we have $\max_{\xi\in[t_j, t_{j+1}]}k'(\xi) = \alpha(t_n-t_{j+1})^{-\alpha-1}$ and $|k(s) - k(t_{j+1})| \leq\tau\alpha(t_n-t_{j+1})^{-\alpha-1}$. Hence, we have
\begin{align}
    \left|\int_{t_j}^{t_{j+1}} (t_n - s)^{-\alpha} (y'(s) - u_j) ds \right| &\le \int_{t_j}^{t_{j+1}} |(t_n - s)^{-\alpha} - (t_n - t_{j+1})^{-\alpha}| |y'(s) - u_j| ds \nonumber \\
    &\le \alpha \tau^2 (t_n - t_{j+1})^{-\alpha-1} Y_j.
\end{align}
From H\"older's inequality with conjugates $p, q$ ($1/p + 1/q = 1$), we have
\begin{align}
    Y_j &= \int_{t_j}^{t_{j+1}} |y''(s)| \omega(s)^{1/p} \omega(s)^{-1/p} ds \nonumber \\
    &\le \left( \int_{t_j}^{t_{j+1}} |y''(s)|^p \omega(s) ds \right)^{1/p} \left( \int_{t_j}^{t_{j+1}} \omega(s)^{-q/p} ds \right)^{1/q} \nonumber \\
    &= \|y''\|_{L^p_\omega(t_j, t_{j+1})} \|\omega^{-1/p}\|_{L^q(t_j, t_{j+1})}.
\end{align}
Let us denote $W_j := \|\omega^{-1/p}\|_{L^q(t_j, t_{j+1})}$. Plugging all these calculations into \eqref{eq:error_rep}, we have
\begin{equation}
    R_n(y) \le C \left[ \tau^{1-\alpha} W_{n-1} \|y''\|_{L^p_\omega(t_{n-1}, t_{n})} + \sum_{j=0}^{n-2} \tau^2 (t_n - t_{j+1})^{-\alpha-1} W_j \|y''\|_{L^p_\omega(t_j, t_{j+1})} \right].
\end{equation}
Applying H\"older's inequality to the sum in the above inequality, we have
\begin{equation}
    \sum_{j=0}^{n-2} A_j B_j \le \left( \sum_{j=0}^{n-2} A_j^q \right)^{1/q} \left( \sum_{j=0}^{n-2} B_j^p \right)^{1/p},
\end{equation}
where $A_j =\tau^2 (t_n - t_{j+1})^{-\alpha-1} W_j$ and $B_j = \|y''\|_{L^p_\omega(t_j, t_{j+1})}$. Note that $(\sum B_j^p)^{1/p} = \|y''\|_{L^p_\omega(0,t_n)} \le \|y\|_{W^{2,p}_\omega}$. Now, let us analyze the term:
\begin{equation}
    \left( \sum_{j=0}^{n-2} A_j^q \right)^{1/q} = \left( \sum_{j=0}^{n-2} \left[ \tau^2 (t_n - t_{j+1})^{-\alpha-1} W_j \right]^q \right)^{1/q}.
\end{equation}
To do this, we introduce
\begin{equation}
    \Lambda(\omega, \tau) := \sup_{0 \le j \le N-1} \tau^{-1/q} W_j = \sup_{j} \left( \frac{1}{\tau} \int_{t_j}^{t_{j+1}} \omega(s)^{-\frac{1}{p-1}} ds \right)^{\frac{p-1}{p}}.
\end{equation}
Then $W_j \le \tau^{1/q} \Lambda(\omega, \tau)$. Thus, we have
\begin{align}
    \left( \sum_{j=0}^{n-2} A_j^q \right)^{1/q} &\le \tau^2 \tau^{1/q} \Lambda(\omega, \tau) \left( \sum_{j=0}^{n-2} (t_n - t_{j+1})^{-(\alpha+1)q} \right)^{1/q} \nonumber \\
    &\le \tau^2 \tau^{1/q} \Lambda(\omega, \tau) \left( \frac{1}{\tau} \int_0^{t_n-\tau} (t_n - s)^{-(\alpha+1)q} ds \right)^{1/q} \nonumber \\
    &\le \tau^2 \Lambda(\omega, \tau) \left( \int_0^{t_n-\tau} (t_n - s)^{-(\alpha+1)q} ds \right)^{1/q} \nonumber \\
    &\le C_2 \Lambda(\omega, \tau) \tau^{2-(\alpha+1) + 1/q} = C_2 \Lambda(\omega, \tau) \tau^{2-\alpha-1/p}.
\end{align}
This concludes the proof.
\end{proof}
\begin{corollary}\label{cor1}
    Let $\omega(t) = t^\mu$. If $y \in W^{2,p}_{t^\mu}(0,T)$ and $0 \le \mu < p-1$, the discretization error is bounded by
    \begin{equation}
        |D^\alpha y - \delta_\tau^\alpha y| \le C \tau^{2 - \alpha - \frac{1+\mu}{p}} \|y\|_{W^{2,p}_{t^\mu}}.
    \end{equation}
\end{corollary}
\begin{proof}
    We calculate $\Lambda(\omega, \tau)$ for $\omega(t) = t^\mu$. We must evaluate $W_j = \|t^{-\mu/p}\|_{L^q(t_j, t_{j+1})}$. For the first interval $[0, \tau]$, we have
    \begin{equation}
        W_0 = \left( \int_0^\tau t^{-\frac{\mu q}{p}} dt \right)^{1/q}.
    \end{equation}
    Convergence requires $\frac{\mu q}{p} < 1 \iff \mu < \frac{p}{q} = p-1$. Evaluating the integral, we have
    \begin{equation}
        W_0 = \left( \frac{\tau^{1 - \mu q/p}}{1 - \mu q/p} \right)^{1/q} = \left(\frac{p}{p-\mu q}\right)^{1/q} \tau^{\frac{1}{q} - \frac{\mu}{p}}.
    \end{equation}
    Since $t^{-\mu q/p}$ is a strictly decreasing function of $t > 0$, the supremum is achieved at $j=0$. Thus, we have
    \begin{equation}
        \Lambda(t^\mu, \tau) = \left(\frac{p}{p-\mu q}\right)^{1/q} \tau^{-1/q} \tau^{\frac{1}{q} - \frac{\mu}{p}} = \left(\frac{p}{p-\mu q}\right)^{1/q} \tau^{-\frac{\mu}{p}}.
    \end{equation}
    Finally, we substitute $\Lambda(t^\mu, \tau)$ into the bound obtained in Theorem \ref{th1}, which yields the desired result.
\end{proof}
\begin{corollary}\label{cor2}
Let the Jacobi weight be given by $\omega(t) = t^\mu(T-t)^\gamma$. If $y \in W^{2,p}_\omega(0,T)$ with $0 \le \mu < p-1$ and $0 \le \gamma < p-1$, the discretization error is bounded by
\begin{equation}
    |D^\alpha y(t_n) - \delta_\tau^\alpha y(t_n)| \le C \tau^{2 - \alpha - \frac{1+\nu}{p}} \|y\|_{W^{2,p}_\omega},
\end{equation}
where $\nu = \max(\mu, \gamma)$ and $C$ is a constant independent of $\tau$.
\end{corollary}

\begin{proof}
We need to estimate $\Lambda(\omega, \tau)$ for the Jacobi weight $\omega(t) = t^\mu(T-t)^\gamma$. We must evaluate the integrals 
\begin{equation}
    W_j^q = \int_{t_j}^{t_{j+1}} s^{-\frac{\mu q}{p}} (T-s)^{-\frac{\gamma q}{p}} ds
\end{equation}
for $j = 0, \dots, N-1$. Assume, without loss of generality, that $\tau \le T/2$. We divide the grid intervals into three cases: the left boundary ($j=0$), the right boundary ($j=N-1$), and the interior intervals ($1 \le j \le N-2$). For the left boundary $j=0$, we have $s \in [0, \tau]$. Thus, $(T-s) \ge T/2$. This allows us to bound the integral as follows
\begin{equation}
    W_0^q \le (T/2)^{-\frac{\gamma q}{p}} \int_0^\tau s^{-\frac{\mu q}{p}} ds = C_1 \tau^{1 - \frac{\mu q}{p}},
\end{equation}
which converges since $\mu < p-1 \iff \frac{\mu q}{p} < 1$.

For the right boundary $j=N-1$, we have $s \in [T-\tau, T]$, which implies $s \ge T/2$. Similarly, we obtain
\begin{equation}
    W_{N-1}^q \le (T/2)^{-\frac{\mu q}{p}} \int_{T-\tau}^T (T-s)^{-\frac{\gamma q}{p}} ds = C_2 \tau^{1 - \frac{\gamma q}{p}},
\end{equation}
which converges since $\gamma < p-1$. For the interior intervals ($1 \le j \le N-2$), both $s$ and $T-s$ are bounded away from zero by at least $\tau$, ensuring that the singularity is avoided, and $W_j^q \le C_3 \tau$. 

Since we are concerned with the limit as $\tau \to 0$, and the fractional exponents satisfy $1 - \frac{\mu q}{p} < 1$ and $1 - \frac{\gamma q}{p} < 1$, the boundary terms dominate the interior terms. Taking the $q$-th root, we obtain that for all $j$,
\begin{equation}
    W_j \le C \max\left( \tau^{\frac{1}{q} - \frac{\mu}{p}}, \tau^{\frac{1}{q} - \frac{\gamma}{p}} \right) = C \tau^{\frac{1}{q}} \tau^{-\frac{\max(\mu, \gamma)}{p}}.
\end{equation}
Letting $\nu = \max(\mu, \gamma)$, we calculate the supremum
\begin{equation}
    \Lambda(\omega, \tau) = \sup_j \tau^{-1/q} W_j \le C \tau^{-\frac{\nu}{p}}.
\end{equation}
Substituting this bound into Theorem \ref{th1} yields the final estimate.
\end{proof}

\begin{corollary}\label{cor3}
Let the weight be given by $\omega(t) = \left(\ln \frac{eT}{t}\right)^{-\mu}$ with $\mu > 0$. If $y \in W^{2,p}_\omega(0,T)$, the discretization error is bounded by
\begin{equation}
    |D^\alpha y(t_n) - \delta_\tau^\alpha y(t_n)| \le C \tau^{2 - \alpha - 1/p} \left(\ln \frac{eT}{\tau}\right)^{\frac{\mu}{p}} \|y\|_{W^{2,p}_\omega},
\end{equation}
where $C$ is a constant independent of $\tau$.
\end{corollary}

\begin{proof}
We must evaluate $\Lambda(\omega, \tau)$ for the logarithmic weight $\omega(t) = \left(\ln \frac{eT}{t}\right)^{-\mu}$. To do this, we calculate $W_j = \|\omega^{-1/p}\|_{L^q(t_j, t_{j+1})}$, which requires evaluating the integrals
\begin{equation}
    W_j^q = \int_{t_j}^{t_{j+1}} \left(\ln \frac{eT}{s}\right)^{\frac{\mu q}{p}} ds.
\end{equation}
Because the integrand $\left(\ln \frac{eT}{s}\right)^{\frac{\mu q}{p}}$ is a strictly decreasing function of $s$ in the interval $(0, T]$, the supremum of these integrals over any subinterval of length $\tau$ occurs at the first interval $[0, \tau]$. Thus, we have
\begin{equation}
    W_0^q = \int_0^\tau \left(\ln \frac{eT}{s}\right)^{\frac{\mu q}{p}} ds.
\end{equation}
To evaluate this, we use the substitution $s = \tau x$, which gives $ds = \tau dx$. The integral becomes
\begin{equation}
    W_0^q = \tau \int_0^1 \left(\ln \frac{eT}{\tau x}\right)^{\frac{\mu q}{p}} dx = \tau \int_0^1 \left(\ln \frac{eT}{\tau} - \ln x\right)^{\frac{\mu q}{p}} dx.
\end{equation}
Let $\beta = \frac{\mu q}{p} > 0$. Using the standard algebraic inequality $(A+B)^\beta \le C_\beta(A^\beta + B^\beta)$ for positive $A$ and $B$, we can bound the integral as follows
\begin{equation}
    W_0^q \le C_\beta \tau \int_0^1 \left[ \left(\ln \frac{eT}{\tau}\right)^\beta + (-\ln x)^\beta \right] dx.
\end{equation}
Note that the integral of the second term is the Gamma function, $\int_0^1 (-\ln x)^\beta dx = \Gamma(\beta+1)$, which is a finite constant. Therefore, we obtain
\begin{equation}
    W_0^q \le C_\beta \tau \left(\ln \frac{eT}{\tau}\right)^\beta + C_\beta \tau \Gamma(\beta+1).
\end{equation}
For sufficiently small step sizes $\tau$, the term containing the logarithm dominates. Hence, there exists a constant $C > 0$ such that:
\begin{equation}
    W_0^q \le C \tau \left(\ln \frac{eT}{\tau}\right)^{\frac{\mu q}{p}}.
\end{equation}
Taking the $q$-th root yields
\begin{equation}
    W_0 \le C^{1/q} \tau^{1/q} \left(\ln \frac{eT}{\tau}\right)^{\frac{\mu}{p}}.
\end{equation}
Finally, we compute the supremum term
\begin{equation}
    \Lambda(\omega, \tau) = \sup_j \tau^{-1/q} W_j = \tau^{-1/q} W_0 \le \tilde{C} \left(\ln \frac{eT}{\tau}\right)^{\frac{\mu}{p}}.
\end{equation}
Substituting this bound for $\Lambda(\omega, \tau)$ into the general error estimate in Theorem \ref{th1} yields the claim.
\end{proof}

\section{Application to a Fractional Differential Equation}
To demonstrate the versatility of our truncation error bound, we apply the L1 scheme to a standard linear fractional initial value problem
\begin{equation}\label{eq:fode}
    D^\alpha y(t) + \lambda y(t) = f(t), \quad t \in (0,T],
\end{equation}
with the initial condition $y(0) = y_0$, where $\lambda \ge 0$ and $f$ is a source term assumed to be continuous. The corresponding numerical scheme is obtained by evaluating \eqref{eq:fode} at $t_n$ and replacing the Caputo derivative with the L1 approximation $\delta_\tau^\alpha$. That is,
\begin{equation} \label{eq:fode_discrete}
    \delta_\tau^\alpha Y_n + \lambda Y_n = f(t_n), \quad n = 1, 2, \dots, N,
\end{equation}
with $Y_0 = y_0$. Here, $Y_n$ denotes the numerical approximation of $y(t_n)$. Set $e_n = y(t_n) - Y_n$ to represent the global discretization error.

Combining the discrete stability of the L1 scheme with our main result, we obtain the following global convergence bound.
\begin{theorem}[Convergence for the Fractional ODE]
Suppose that the exact solution to \eqref{eq:fode} satisfies $y \in W^{2,p}_\omega(0,T)$ with $\omega \in A_p$. Then, there exists a constant $C > 0$, independent of $\tau$, such that the global error satisfies
\begin{equation}
    \max_{1 \le n \le N} |e_n| \le C \tau^{2-\alpha-1/p} \Lambda(\omega, \tau) \|y\|_{W^{2,p}_\omega}.
\end{equation}
\end{theorem}
\begin{proof}
Evaluating the exact equation \eqref{eq:fode} at $t_n$ gives
\begin{equation} \label{eq:fode_exact}
    D^\alpha y(t_n) + \lambda y(t_n) = f(t_n).
\end{equation}
Now, we subtract the numerical scheme \eqref{eq:fode_discrete} from \eqref{eq:fode_exact} to obtain
\begin{equation}
    (D^\alpha y(t_n) - \delta_\tau^\alpha Y_n) + \lambda e_n = 0.
\end{equation}
By adding and subtracting $\delta_\tau^\alpha y(t_n)$, and using the linearity of the operator $\delta_\tau^\alpha$, we further have
\begin{equation}
    \delta_\tau^\alpha e_n + \lambda e_n = - R_n(y),
\end{equation}
where $R_n(y) = D^\alpha y(t_n) - \delta_\tau^\alpha y(t_n)$ is the truncation error analyzed in Theorem \ref{th1}. Because $\lambda \ge 0$, it is a known property of the L1 scheme (see, e.g., \cite{LiCai2019}) that the discrete operator is unconditionally stable. Consequently, the global error is governed by the maximum truncation error through the discrete fractional Gr\"onwall inequality, which gives
\begin{equation}
    \max_{1 \le n \le N} |e_n| \le \frac{t_N^\alpha}{\Gamma(1+\alpha)} \max_{1 \le n \le N} |R_n(y)|.
\end{equation}
Applying the uniform bound for $R_n(y)$ obtained in Theorem \ref{th1}, and noting that $t_N = T$, we deduce
\begin{equation}
    \max_{1 \le n \le N} |e_n| \le \frac{T^\alpha}{\Gamma(1+\alpha)} C \tau^{2-\alpha-1/p} \Lambda(\omega, \tau) \|y\|_{W^{2,p}_\omega}.
\end{equation}
Setting $C = \frac{T^\alpha C}{\Gamma(1+\alpha)}$ concludes the proof.
\end{proof}

\section{Numerical experiments}

The order of convergence of the L1 scheme can be estimated via the extrapolation, that is
\begin{equation}
    \text{order} \approx \log_2\frac{\max_{t\in[0,T]}|\delta_\tau^\alpha y(t) - \delta_{\tau/2}^\alpha(t)|}{\max_{t\in[0,T]}|\delta_{\tau/2}^\alpha y(t) - \delta_{\tau/4}^\alpha(t)|}.
\end{equation}
To verify the results of Corollary \ref{cor1}, we choose the following
\begin{equation}\label{eq:test1}
    y_{\kappa}(t) = t^{\kappa},
\end{equation}
where $\kappa = 2 - \alpha - \frac{1+\mu}{p} + 0.001$. In Table \ref{tab:tab1} we have gathered the results of our computations rounded to the fourth decimal place. We immediately observe that the estimated order of convergence agrees with the theoretical value $2 - \alpha - \frac{1+\mu}{p} $ exceptionally well for all the choices of $\alpha$, $p$, and $\mu$.

To check the results of Corollary \ref{cor2}, we choose the following
\begin{equation}\label{eq:test2}
    y_{\rho_0, \rho_T }(t) = t^{\rho_0} + (T-t)^{\rho_T} - T^{\rho_T},
\end{equation}
where $\rho_0 = 2 - \frac{1+\mu}{p} + 0.001$, $\rho_T = 2 - \frac{1+\gamma}{p} + 0.001$. In Table \ref{tab:tab2} we have gathered the results of our computations rounded to the fourth decimal place. We immediately observe that the estimated order of convergence agrees with the theoretical value $2 - \alpha - \frac{1+\mu}{p}$ exceptionally well for all the choices of $\alpha$, $p$, $\mu$ and $\gamma$.

To verify the results of Corollary \ref{cor3}, we choose the following
\begin{equation}\label{eq:test3}
    y_{\rho, \theta}(t) = t^{\rho}\left(\ln\frac{eT}{t}\right)^\theta
\end{equation}
where $\rho= 2 - \frac{1}{p}$, $\theta =\frac{\mu-1}{p} - 0.001$. In Table \ref{tab:tab3} we have gathered the results of our computations rounded to the fourth decimal place. We immediately observe that the estimated order of convergence agrees well with the theoretical value $2 - \alpha - \frac{1+\mu}{p} $ for not all choices of $\alpha$, $p$, and $\mu$. To solve this issue, we also estimate order adjusted to logarithmic prefactor using formula
\begin{equation}\label{eq:logadjusted}
\text{order}_{\text{log}} \approx \log_2 \left( \frac{\max_{t\in[0,T]}|\delta_\tau^\alpha y(t) - \delta_{\tau/2}^\alpha y(t)|}{\max_{t\in[0,T]}|\delta_{\tau/2}^\alpha y(t) - \delta_{\tau/4}^\alpha y(t)|} \cdot \left[ \frac{\ln\left(\frac{2eT}{\tau}\right)}{\ln\left(\frac{eT}{\tau}\right)} \right]^{\frac{\mu}{p}} \right).
\end{equation}
In Table \ref{tab:tab4} we have gathered the results of our computations rounded to the fourth decimal place. We observe that the results are better in problematic cases.
\begin{table}[h]
\centering
\begin{tabular}{cccccc}
\hline
$\alpha$ & $p$ & $\mu$ & $\kappa$ & \textbf{Estimated order} & \textbf{Theoretical order} \\
 & & & &  & $(2 - \alpha - \frac{1+\mu}{p})$ \\
\hline
0.1 & 1.5 & 0.00 & 1.334 & \textbf{1.234} & 1.233 \\
0.1 & 1.5 & 0.40 & 1.068 & \textbf{0.968} & 0.967 \\
0.1 & 3.0 & 0.00 & 1.668 & \textbf{1.568} & 1.567 \\
0.1 & 3.0 & 1.60 & 1.134 & \textbf{1.034} & 1.033 \\
\hline
0.3 & 1.5 & 0.00 & 1.334 & \textbf{1.034} & 1.033 \\
0.3 & 1.5 & 0.40 & 1.068 & \textbf{0.768} & 0.767 \\
0.3 & 3.0 & 0.00 & 1.668 & \textbf{1.368} & 1.367 \\
0.3 & 3.0 & 1.60 & 1.134 & \textbf{0.834} & 0.833 \\
\hline
0.5 & 1.5 & 0.00 & 1.334 & \textbf{0.834} & 0.833 \\
0.5 & 1.5 & 0.40 & 1.068 & \textbf{0.568} & 0.567 \\
0.5 & 3.0 & 0.00 & 1.668 & \textbf{1.168} & 1.167 \\
0.5 & 3.0 & 1.60 & 1.134 & \textbf{0.634} & 0.633 \\
\hline
0.6 & 1.5 & 0.00 & 1.334 & \textbf{0.734} & 0.733 \\
0.6 & 1.5 & 0.40 & 1.068 & \textbf{0.468} & 0.467 \\
0.6 & 3.0 & 0.00 & 1.668 & \textbf{1.068} & 1.067 \\
0.6 & 3.0 & 1.60 & 1.134 & \textbf{0.534} & 0.533 \\
\hline
0.9 & 1.5 & 0.00 & 1.334 & \textbf{0.434} & 0.433 \\
0.9 & 1.5 & 0.40 & 1.068 & \textbf{0.168} & 0.167 \\
0.9 & 3.0 & 0.00 & 1.668 & \textbf{0.768} & 0.767 \\
0.9 & 3.0 & 1.60 & 1.134 & \textbf{0.234} & 0.233 \\
\hline
\end{tabular}
\caption{Estimated orders of convergence of the L1 scheme for various choices of $\alpha$, $p$ and $\mu$ with $T = 1$. The test function is taken to be \eqref{eq:test1}. The base of calculations is $\tau = 2^{-10}$}
\label{tab:tab1}
\end{table}

\begin{table}[h]
\centering
\begin{tabular}{ccccccccc}
\hline
$\alpha$ & $p$ & $\mu$ & $\gamma$ & $\rho_0$ & $\rho_T$ & $\nu = \max(\mu,\gamma)$ & \textbf{Estimated order} & \textbf{Theoretical order} \\
 & & & & & & & & $(2 - \alpha - \frac{1+\nu}{p})$ \\
\hline
0.1 & 1.5 & 0.00 & 0.00 & 1.334 & 1.334 & 0.00 & \textbf{1.238} & 1.233 \\
0.1 & 1.5 & 0.40 & 0.10 & 1.068 & 1.268 & 0.40 & \textbf{0.969} & 0.967 \\
0.1 & 1.5 & 0.10 & 0.40 & 1.268 & 1.068 & 0.40 & \textbf{0.970} & 0.967 \\
0.1 & 1.5 & 0.25 & 0.25 & 1.168 & 1.168 & 0.25 & \textbf{1.069} & 1.067 \\
0.1 & 3.0 & 0.00 & 0.00 & 1.668 & 1.668 & 0.00 & \textbf{1.589} & 1.567 \\
0.1 & 3.0 & 1.60 & 0.40 & 1.134 & 1.534 & 1.60 & \textbf{1.036} & 1.033 \\
0.1 & 3.0 & 0.40 & 1.60 & 1.534 & 1.134 & 1.60 & \textbf{1.039} & 1.033 \\
0.1 & 3.0 & 1.00 & 1.00 & 1.334 & 1.334 & 1.00 & \textbf{1.238} & 1.233 \\
\hline
0.3 & 1.5 & 0.00 & 0.00 & 1.334 & 1.334 & 0.00 & \textbf{1.037} & 1.033 \\
0.3 & 1.5 & 0.40 & 0.10 & 1.068 & 1.268 & 0.40 & \textbf{0.769} & 0.767 \\
0.3 & 1.5 & 0.10 & 0.40 & 1.268 & 1.068 & 0.40 & \textbf{0.769} & 0.767 \\
0.3 & 1.5 & 0.25 & 0.25 & 1.168 & 1.168 & 0.25 & \textbf{0.868} & 0.867 \\
0.3 & 3.0 & 0.00 & 0.00 & 1.668 & 1.668 & 0.00 & \textbf{1.388} & 1.367 \\
0.3 & 3.0 & 1.60 & 0.40 & 1.134 & 1.534 & 1.60 & \textbf{0.836} & 0.833 \\
0.3 & 3.0 & 0.40 & 1.60 & 1.534 & 1.134 & 1.60 & \textbf{0.838} & 0.833 \\
0.3 & 3.0 & 1.00 & 1.00 & 1.334 & 1.334 & 1.00 & \textbf{1.037} & 1.033 \\
\hline
0.5 & 1.5 & 0.00 & 0.00 & 1.334 & 1.334 & 0.00 & \textbf{0.837} & 0.833 \\
0.5 & 1.5 & 0.40 & 0.10 & 1.068 & 1.268 & 0.40 & \textbf{0.569} & 0.567 \\
0.5 & 1.5 & 0.10 & 0.40 & 1.268 & 1.068 & 0.40 & \textbf{0.569} & 0.567 \\
0.5 & 1.5 & 0.25 & 0.25 & 1.168 & 1.168 & 0.25 & \textbf{0.668} & 0.667 \\
0.5 & 3.0 & 0.00 & 0.00 & 1.668 & 1.668 & 0.00 & \textbf{1.187} & 1.167 \\
0.5 & 3.0 & 1.60 & 0.40 & 1.134 & 1.534 & 1.60 & \textbf{0.636} & 0.633 \\
0.5 & 3.0 & 0.40 & 1.60 & 1.534 & 1.134 & 1.60 & \textbf{0.637} & 0.633 \\
0.5 & 3.0 & 1.00 & 1.00 & 1.334 & 1.334 & 1.00 & \textbf{0.837} & 0.833 \\
\hline
0.6 & 1.5 & 0.00 & 0.00 & 1.334 & 1.334 & 0.00 & \textbf{0.736} & 0.733 \\
0.6 & 1.5 & 0.40 & 0.10 & 1.068 & 1.268 & 0.40 & \textbf{0.469} & 0.467 \\
0.6 & 1.5 & 0.10 & 0.40 & 1.268 & 1.068 & 0.40 & \textbf{0.469} & 0.467 \\
0.6 & 1.5 & 0.25 & 0.25 & 1.168 & 1.168 & 0.25 & \textbf{0.568} & 0.567 \\
0.6 & 3.0 & 0.00 & 0.00 & 1.668 & 1.668 & 0.00 & \textbf{1.086} & 1.067 \\
0.6 & 3.0 & 1.60 & 0.40 & 1.134 & 1.534 & 1.60 & \textbf{0.536} & 0.533 \\
0.6 & 3.0 & 0.40 & 1.60 & 1.534 & 1.134 & 1.60 & \textbf{0.537} & 0.533 \\
0.6 & 3.0 & 1.00 & 1.00 & 1.334 & 1.334 & 1.00 & \textbf{0.736} & 0.733 \\
\hline
0.9 & 1.5 & 0.00 & 0.00 & 1.334 & 1.334 & 0.00 & \textbf{0.436} & 0.433 \\
0.9 & 1.5 & 0.40 & 0.10 & 1.068 & 1.268 & 0.40 & \textbf{0.169} & 0.167 \\
0.9 & 1.5 & 0.10 & 0.40 & 1.268 & 1.068 & 0.40 & \textbf{0.169} & 0.167 \\
0.9 & 1.5 & 0.25 & 0.25 & 1.168 & 1.168 & 0.25 & \textbf{0.268} & 0.267 \\
0.9 & 3.0 & 0.00 & 0.00 & 1.668 & 1.668 & 0.00 & \textbf{0.785} & 0.767 \\
0.9 & 3.0 & 1.60 & 0.40 & 1.134 & 1.534 & 1.60 & \textbf{0.236} & 0.233 \\
0.9 & 3.0 & 0.40 & 1.60 & 1.534 & 1.134 & 1.60 & \textbf{0.236} & 0.233 \\
0.9 & 3.0 & 1.00 & 1.00 & 1.334 & 1.334 & 1.00 & \textbf{0.436} & 0.433 \\
\hline
\end{tabular}
\caption{Estimated orders of convergence of the L1 scheme for various choices of $\alpha$, $p$, $\mu$ and $\gamma$ with $T = 1$. The test function is taken to be \eqref{eq:test2}. The base of calculations is $\tau = 2^{-10}$.}
\label{tab:tab2}
\end{table}

\begin{table}[h]
\centering
\begin{tabular}{ccccccc}
\hline
$\alpha$ & $p$ & $\mu$ & $\rho$ & $\theta$ & \textbf{Estimated order} & \textbf{Theoretical order} \\
 & & & &  &  & $(2 - \alpha - 1/p)$ \\
\hline
0.1 & 1.5 & 0.5 & 1.333 & -0.334 & \textbf{1.280} & 1.233 \\
0.1 & 1.5 & 2.0 & 1.333 & 0.666 & \textbf{1.132} & 1.233 \\
0.1 & 3.0 & 0.5 & 1.667 & -0.168 & \textbf{1.588} & 1.567 \\
0.1 & 3.0 & 2.0 & 1.667 & 0.332 & \textbf{1.525} & 1.567 \\
\hline
0.3 & 1.5 & 0.5 & 1.333 & -0.334 & \textbf{1.080} & 1.033 \\
0.3 & 1.5 & 2.0 & 1.333 & 0.666 & \textbf{0.932} & 1.033 \\
0.3 & 3.0 & 0.5 & 1.667 & -0.168 & \textbf{1.388} & 1.367 \\
0.3 & 3.0 & 2.0 & 1.667 & 0.332 & \textbf{1.325} & 1.367 \\
\hline
0.5 & 1.5 & 0.5 & 1.333 & -0.334 & \textbf{0.880} & 0.833 \\
0.5 & 1.5 & 2.0 & 1.333 & 0.666 & \textbf{0.732} & 0.833 \\
0.5 & 3.0 & 0.5 & 1.667 & -0.168 & \textbf{1.188} & 1.167 \\
0.5 & 3.0 & 2.0 & 1.667 & 0.332 & \textbf{1.125} & 1.167 \\
\hline
0.6 & 1.5 & 0.5 & 1.333 & -0.334 & \textbf{0.780} & 0.733 \\
0.6 & 1.5 & 2.0 & 1.333 & 0.666 & \textbf{0.632} & 0.733 \\
0.6 & 3.0 & 0.5 & 1.667 & -0.168 & \textbf{1.088} & 1.067 \\
0.6 & 3.0 & 2.0 & 1.667 & 0.332 & \textbf{1.025} & 1.067 \\
\hline
0.9 & 1.5 & 0.5 & 1.333 & -0.334 & \textbf{0.480} & 0.433 \\
0.9 & 1.5 & 2.0 & 1.333 & 0.666 & \textbf{0.332} & 0.433 \\
0.9 & 3.0 & 0.5 & 1.667 & -0.168 & \textbf{0.788} & 0.767 \\
0.9 & 3.0 & 2.0 & 1.667 & 0.332 & \textbf{0.725} & 0.767 \\
\hline
\end{tabular}
\caption{Estimated orders of convergence of the L1 scheme for various choices of $\alpha$, $p$ and $\mu$ with $T = 1$. The test function is taken to be \eqref{eq:test3}. The base of calculations is $\tau = 2^{-10}$.}
\label{tab:tab3}
\end{table}

\begin{table}[h]
\centering
\begin{tabular}{ccccccc}
\hline
$\alpha$ & $p$ & $\mu$ & $\rho$ & $\theta$ & \textbf{Estimated order} & \textbf{Theoretical order} \\
 & & & & & & $(2 - \alpha - 1/p)$ \\
\hline
0.1 & 1.5 & 0.5 & 1.333 & -0.334 & \textbf{1.317} & 1.233 \\
0.1 & 1.5 & 2.0 & 1.333 & 0.666 & \textbf{1.281} & 1.233 \\
0.1 & 3.0 & 0.5 & 1.667 & -0.168 & \textbf{1.606} & 1.567 \\
0.1 & 3.0 & 2.0 & 1.667 & 0.332 & \textbf{1.599} & 1.567 \\
\hline
0.3 & 1.5 & 0.5 & 1.333 & -0.334 & \textbf{1.117} & 1.033 \\
0.3 & 1.5 & 2.0 & 1.333 & 0.666 & \textbf{1.081} & 1.033 \\
0.3 & 3.0 & 0.5 & 1.667 & -0.168 & \textbf{1.406} & 1.367 \\
0.3 & 3.0 & 2.0 & 1.667 & 0.332 & \textbf{1.399} & 1.367 \\
\hline
0.5 & 1.5 & 0.5 & 1.333 & -0.334 & \textbf{0.917} & 0.833 \\
0.5 & 1.5 & 2.0 & 1.333 & 0.666 & \textbf{0.881} & 0.833 \\
0.5 & 3.0 & 0.5 & 1.667 & -0.168 & \textbf{1.206} & 1.167 \\
0.5 & 3.0 & 2.0 & 1.667 & 0.332 & \textbf{1.199} & 1.167 \\
\hline
0.6 & 1.5 & 0.5 & 1.333 & -0.334 & \textbf{0.817} & 0.733 \\
0.6 & 1.5 & 2.0 & 1.333 & 0.666 & \textbf{0.781} & 0.733 \\
0.6 & 3.0 & 0.5 & 1.667 & -0.168 & \textbf{1.106} & 1.067 \\
0.6 & 3.0 & 2.0 & 1.667 & 0.332 & \textbf{1.099} & 1.067 \\
\hline
0.9 & 1.5 & 0.5 & 1.333 & -0.334 & \textbf{0.517} & 0.433 \\
0.9 & 1.5 & 2.0 & 1.333 & 0.666 & \textbf{0.481} & 0.433 \\
0.9 & 3.0 & 0.5 & 1.667 & -0.168 & \textbf{0.806} & 0.767 \\
0.9 & 3.0 & 2.0 & 1.667 & 0.332 & \textbf{0.799} & 0.767 \\
\hline
\end{tabular}
\caption{Estimated orders of convergence of the L1 scheme using formula \eqref{eq:logadjusted} for various choices of $\alpha$, $p$ and $\mu$ with $T = 1$. The test function is taken to be \eqref{eq:test3}. The base of calculations is $\tau = 2^{-10}$.}
\label{tab:tab4}
\end{table}

\section{Conclusion}

In this paper, we have presented a rigorous discretization error analysis for the L1 approximation of the Caputo fractional derivative. By conducting our analysis within the framework of Muckenhoupt $A_p$ weighted Sobolev spaces, $W^{2,p}_\omega(0,T)$, we successfully circumvented the unrealistic assumption of full $C^2[0,T]$ or $W^{2,p}(0,T)$ regularity that is frequently used in the classical literature. Our results provide sharp, uniform truncation error bounds that explicitly incorporate the weight function, allowing the numerical analysis to gracefully handle the weak singularities naturally inherent to time-fractional differential equations. Furthermore, we demonstrated the flexibility of this theoretical framework by explicitly deriving convergence rates for various physically relevant singularity profiles—including initial algebraic, terminal algebraic, Jacobi, and logarithmic weights—and by establishing a global convergence bound for a representative fractional initial value problem.

Although the current framework provides robust estimates for solutions whose singular behavior can be controlled to ensure $y \in W^{2,p}_\omega(0,T)$, exact solutions to fractional partial differential equations frequently possess intermediate regularity. Future work will focus on extending these weighted error estimates to functions in fractional-order Sobolev spaces $W^{s,p}_\omega(0,T)$ with $1 < s < 2$.

\section*{Acknowledgment}
This research was supported by the National Science Centre, Poland (NCN) under the grant Sonata Bis with a number NCN 2020/38/E/ST1/00153.

\bibliographystyle{plain}
\bibliography{bibliography.bib}
\end{document}